\documentclass[proceedings]{aofa}

\usepackage[T1]{fontenc}
\usepackage[utf8]{inputenc}
\usepackage{subfigure}
\usepackage{hyperref}
\usepackage{cite}
\usepackage{amsmath}

\usepackage{amsthm}

\newcommand{\M}{\mathcal{M}}
\newcommand{\B}{\mathcal{B}}

\newcommand{\hn}{h_{\mathsf{NAF}}}
\newcommand{\hg}{h_{\mathsf{GRAY}}}
\newcommand{\ro}{r_{\mathsf{OPT}}}
\DeclareMathOperator{\rank}{rank}
\newcommand{\bff}{\boldsymbol{f}}
\newcommand{\bfu}{\boldsymbol{u}}
\newcommand{\bfv}{\boldsymbol{v}}
\newcommand{\bfx}{\boldsymbol{x}}
\newcommand{\bfy}{\boldsymbol{y}}

\newif\ifdetails
\detailstrue
\newcommand{\DETAIL}[1]%
{\ifdetails\par\fbox{\begin{minipage}{0.9\linewidth}\textit{Detail:}
      #1\end{minipage}}\par\fi}
\newcommand{\TODO}[1]%
{\ifdetails\par\fbox{\begin{minipage}{0.9\linewidth}\textbf{TODO:}
      #1\end{minipage}}\par\fi}

\newtheorem{lemma}{Lemma}
\newtheorem{prop}[lemma]{Proposition}
\newtheorem{theorem}[lemma]{Theorem}
\newtheorem{cor}[lemma]{Corollary}
\theoremstyle{remark}
\newtheorem{example}{Example}
\newtheorem*{remark}{Remark}
\newtheorem*{defi}{Definition}

\title{$q$-Quasiadditive Functions}

\author[Sara Kropf \and Stephan Wagner]{Sara Kropf\addressmark{1,2}\thanks{The first author is supported by the Austrian Science Fund (FWF): P~24644-N26.}\ \addressmark{\S} \and Stephan Wagner\addressmark{3}\thanks{The second author is supported by the National Research Foundation of South Africa under grant number 96236.}
\thanks{The authors were also supported by the Karl Popper Kolleg
``Modeling--Simulation--Optimization'' funded by the Alpen-Adria-Universit\"at
Klagenfurt and by the Carinthian Economic Promotion Fund (KWF). Part of this paper was written while the second author was a Karl Popper Fellow at the Mathematics Institute in Klagenfurt. He would like to thank the institute for the hospitality received.}}
\address{\addressmark{1}Institut f\"ur Mathematik, Alpen-Adria-Universit\"at
  Klagenfurt, Austria, sara.kropf@aau.at\\
\addressmark{2}Institute of Statistical Science, Academia
  Sinica, Taipei, Taiwan, sarakropf@stat.sinica.edu.tw\\
\addressmark{3}Department of Mathematical Sciences, Stellenbosch University, South
  Africa, swagner@sun.ac.za}

\begin{document}
%\publicationdetails{VOL}{2016}{ISS}{NUM}{SUBM}
\maketitle

\begin{abstract}
In this paper, we introduce the notion of \begin{math}q\end{math}-quasiadditivity of
arithmetic functions, as well as the related concept of
\begin{math}q\end{math}-quasimultiplicativity, which generalises strong \begin{math}q\end{math}-additivity
and -multiplicativity, respectively. We show that there are many
natural examples for these concepts, which are characterised by
functional equations of the form \begin{math}f(q^{k+r}a + b) = f(a) + f(b)\end{math} or
\begin{math}f(q^{k+r}a + b) = f(a) f(b)\end{math} for all \begin{math}b < q^k\end{math} and a fixed parameter \begin{math}r\end{math}.
In addition to some elementary properties of \begin{math}q\end{math}-quasiadditive and \begin{math}q\end{math}-quasimultiplicative functions, we prove characterisations of \begin{math}q\end{math}-quasiadditivity and \begin{math}q\end{math}-quasimultiplicativity for the special class of \begin{math}q\end{math}-regular functions. The final main result provides a general central limit theorem that includes both classical and new examples as corollaries.
\end{abstract}

\section{Introduction}

Arithmetic functions based on the digital expansion in some base \begin{math}q\end{math}
have a long history (see, e.g., \cite{Bellman-Shapiro:1948,Gelfond:1968:sur,Delange:1972:q-add-q-mult,Delange:1975:chiffres,Cateland:digital-seq,Bassily-Katai:1995:distr,Drmota:2000})
 The notion of a \begin{math}q\end{math}-\emph{additive} function is due to \cite{Gelfond:1968:sur}: an arithmetic function (defined on nonnegative integers) is called \begin{math}q\end{math}-additive if
\begin{equation*}f(q^k a + b) = f(q^k a) + f(b)\end{equation*}
whenever \begin{math}0 \leq b < q^k\end{math}. A stronger version of this concept is \emph{strong} (or \emph{complete}) \begin{math}q\end{math}-additivity: a function \begin{math}f\end{math} is said to be strongly \begin{math}q\end{math}-additive if we even have
\begin{equation*}f(q^k a + b) = f(a) + f(b)\end{equation*}
whenever \begin{math}0 \leq b < q^k\end{math}. The class of (strongly) \begin{math}q\end{math}-\emph{multiplicative} functions is defined in an analogous fashion.
Loosely speaking, (strong) \begin{math}q\end{math}-additivity of a function means that
it can be evaluated by breaking up the base-\begin{math}q\end{math} expansion. Typical
examples of strongly \begin{math}q\end{math}-additive functions are the \begin{math}q\end{math}-ary sum of
digits and the number of occurrences of a specified nonzero digit.

There are, however, many simple and natural functions based on the \begin{math}q\end{math}-ary expansion that are not \begin{math}q\end{math}-additive. A very basic example of this kind are \emph{block counts}: the number of occurrences of a certain block of digits in the \begin{math}q\end{math}-ary expansion. This and other examples provide the motivation for the present paper, in which we define and study a larger class of functions with comparable properties.

\begin{defi}
An arithmetic function (a function defined on the set of nonnegative integers) is called \begin{math}q\end{math}-\emph{quasiadditive} if there exists some nonnegative integer \begin{math}r\end{math} such that
\begin{equation}\label{eq:q-add}
f(q^{k+r}a + b) = f(a) + f(b)
\end{equation}
whenever \begin{math}0 \leq b < q^k\end{math}. Likewise, \begin{math}f\end{math} is said to be \begin{math}q\end{math}-\emph{quasimultiplicative} if it satisfies the identity
\begin{equation}\label{eq:q-mult}
f(q^{k+r}a + b) = f(a)f(b)
\end{equation}
for some fixed nonnegative integer \begin{math}r\end{math} whenever \begin{math}0 \leq b < q^k\end{math}.
\end{defi}

We remark that the special case \begin{math}r = 0\end{math} is exactly strong
\begin{math}q\end{math}-additivity, so strictly speaking the term ``strongly
\begin{math}q\end{math}-quasiadditive function'' might be more appropriate. However, since
we are not considering a weaker version (for which natural examples
seem to be much harder to find), we do not make a distinction. As a further caveat, we remark that the term ``quasiadditivity'' has also been used in \cite{allouche:1993} for a related, but slightly weaker condition.

In the
following section, we present a variety of examples of
\begin{math}q\end{math}-quasiadditive and \begin{math}q\end{math}-quasimultipli\-cative functions. 
In Section~\ref{sec:elem-properties}, we give some general properties of
such functions. Since most of our examples also belong to the related class of \begin{math}q\end{math}-regular
functions, we discuss the connection in Section~\ref{sec:q-regular}.
Finally, we prove a general central limit theorem for \begin{math}q\end{math}-quasiadditive and -multiplicative functions that contains both old
and new examples as special cases.

\section{Examples of $q$-quasiadditive and $q$-quasimultiplicative  functions}
\label{sec:exampl-q-quasiadd}
Let us now back up the abstract concept of \begin{math}q\end{math}-quasiadditivity by some concrete examples.

\subsection*{Block counts}

As mentioned in the introduction, the number of occurrences of a fixed nonzero
digit is a typical example of a \begin{math}q\end{math}-additive function. However, the
number of occurrences of a given block \begin{math}B = \epsilon_1\epsilon_2
\cdots \epsilon_{\ell}\end{math} of digits in the expansion of a nonnegative
integer \begin{math}n\end{math}, which we denote by \begin{math}c_B(n)\end{math}, does not represent a
\begin{math}q\end{math}-additive function. The reason is simple: the \begin{math}q\end{math}-ary expansion of \begin{math}q^ka + b\end{math} is obtained by joining the expansions of \begin{math}a\end{math} and \begin{math}b\end{math}, so occurrences of \begin{math}B\end{math} in \begin{math}a\end{math} and occurrences of \begin{math}B\end{math} in \begin{math}b\end{math} are counted by \begin{math}c_B(a) + c_{B}(b)\end{math}, but occurrences that involve digits of both \begin{math}a\end{math} and \begin{math}b\end{math} are not.

However, if \begin{math}B\end{math} is a block different from \begin{math}00\cdots0\end{math}, then \begin{math}c_B\end{math} is \begin{math}q\end{math}-quasiadditive: note that the representation of \begin{math}q^{k+\ell} a + b\end{math} is of the form
\begin{equation*}\underbrace{a_1 a_2 \cdots a_{\mu}}_{\text{expansion of } a} \underbrace{0 0 \cdots 0_{\vphantom{\mu}}}_{\ell \text{ zeros}} \underbrace{b_1 b_2 \cdots {b_{\nu}}_{\vphantom{\mu}}}_{\text{expansion of } b}\end{equation*}
whenever \begin{math}0 \leq b < q^k\end{math}, so occurrences of the block \begin{math}B\end{math} have to belong to either \begin{math}a\end{math} or \begin{math}b\end{math} only. This implies that
\begin{math}c_B(q^{k+\ell} a + b) = c_B(a) + c_B(b)\end{math},
with one small caveat: if the block starts and/or ends with a sequence
of zeros, then the count needs to be adjusted by assuming the digital
expansion of a nonnegative integer to be padded with zeros on the left
and on the right. 

For example, let \begin{math}B\end{math} be the block \begin{math}0101\end{math} in base \begin{math}2\end{math}. The binary representations of \begin{math}469\end{math} and \begin{math}22\end{math} are \begin{math}111010101\end{math} and \begin{math}10110\end{math}, respectively, so we have \begin{math}c_B(469) = 2\end{math} and \begin{math}c_B(22) = 1\end{math} (note the occurrence of \begin{math}0101\end{math} at the beginning of \begin{math}10110\end{math} if we assume the expansion to be padded with zeros), as well as
\begin{equation*}c_B(240150) = c_B(2^9 \cdot 469 + 22) = c_B(469) + c_B(22) = 3.\end{equation*}
Indeed, the block \begin{math}B\end{math} occurs three times in the expansion of \begin{math}240150\end{math}, which is \begin{math}111010101000010110\end{math}.

\subsection*{The number of runs and the Gray code}

The number of ones in the Gray code of a nonnegative integer \begin{math}n\end{math},
which we denote by \begin{math}\hg(n)\end{math}, is also equal to the number of runs
(maximal sequences of consecutive identical digits) in the binary
representations of \begin{math}n\end{math} (counting the number of runs in the
representation of \begin{math}0\end{math} as \begin{math}0\end{math}); the sequence defined by \begin{math}\hg(n)\end{math} is
\href{http://oeis.org/A005811}{A005811} in Sloane's On-Line Encyclopedia of Integer Sequences
\cite{OEIS:2016}. An analysis of its expected value is performed in \cite{Flajolet-Ramshaw:1980:gray}. The function \begin{math}\hg\end{math} is \begin{math}2\end{math}-quasiadditive up to some minor
modification: set \begin{math}f(n) = \hg(n)\end{math} if \begin{math}n\end{math} is even and \begin{math}f(n) = \hg(n)
+ 1\end{math} if \begin{math}n\end{math} is odd. The new function \begin{math}f\end{math} can be interpreted as the
total number of occurrences of the two blocks \begin{math}01\end{math} and \begin{math}10\end{math} in the
binary expansion (considering binary expansions to be padded with zeros at both ends), so the argument of the previous example applies again and shows that \begin{math}f\end{math} is \begin{math}2\end{math}-quasiadditive.

\subsection*{The nonadjacent form and its Hamming weight}

The nonadjacent form (NAF) of a nonnegative integer is the unique
base-\begin{math}2\end{math} representation with digits \begin{math}0,1,-1\end{math} (\begin{math}-1\end{math} is usually
represented as \begin{math}\overline{1}\end{math} in this context) and the additional
requirement that there may not be two adjacent nonzero digits, see
\cite{Reitwiesner:1960}. For example, the NAF of \begin{math}27\end{math} is
\begin{math}100\overline{1}0\overline{1}\end{math}. It is well known that the NAF always
has minimum Hamming weight (i.e., the number of nonzero digits) among all
possible binary representations with this particular digit set,
although it may not be unique with this property (compare, e.g.,
\cite{Reitwiesner:1960} with \cite{Joye-Yen:2000:optim-left}).

The Hamming weight \begin{math}\hn\end{math} of the nonadjacent form has been analysed in
some detail \cite{Thuswaldner:1999,Heuberger-Kropf:2013:analy}, and it is also an example of a \begin{math}2\end{math}-quasiadditive function. It is not difficult to see that \begin{math}\hn\end{math} is characterised by the recursions
$\hn(2n) = \hn(n)$, $\hn(4n+1) = \hn(n) + 1$, $\hn(4n-1) = \hn(n) + 1$
together with the initial value \begin{math}\hn(0) = 0\end{math}. The identity
\begin{equation*}\hn(2^{k+2}a + b) = \hn(a) + \hn(b)\end{equation*}
can be proved by induction. In Section~\ref{sec:q-regular}, this example will be generalised and put into a larger context.

\subsection*{The number of optimal $\{0,1,-1\}$-representations}

As mentioned above, the NAF may not be the only representation with minimum Hamming weight among all possible binary representations with digits \begin{math}0,1,-1\end{math}. The number of optimal representations of a given nonnegative integer \begin{math}n\end{math} is therefore a quantity of interest in its own right. Its average over intervals of the form \begin{math}[0,N)\end{math} was studied by Grabner and Heuberger \cite{Grabner-Heuberger:2006:Number-Optimal}, who also proved that the number \begin{math}\ro(n)\end{math} of optimal representations of \begin{math}n\end{math} can be obtained in the following way:

\begin{lemma}[Grabner--Heuberger \cite{Grabner-Heuberger:2006:Number-Optimal}]\label{lemma:opt-representations-recursion}
Let sequences \begin{math}u_i\end{math} (\begin{math}i=1,2,\ldots,5\end{math}) be given recursively by
\begin{equation*}u_1(0) = u_2(0) = \cdots = u_5(0) = 1, \qquad u_1(1) = u_2(1) = 1,\ u_3(1) = u_4(1) = u_5(1) = 0,\end{equation*}
and
\begin{align*}
u_1(2n) = u_1(n), \qquad & u_1(2n+1) = u_2(n) + u_4(n+1), \\
u_2(2n) = u_1(n), \qquad & u_2(2n+1) = u_3(n), \\
u_3(2n) = u_2(n), \qquad & u_3(2n+1) = 0, \\
u_4(2n) = u_1(n), \qquad & u_4(2n+1) = u_5(n+1), \\
u_5(2n) = u_4(n), \qquad & u_5(2n+1) = 0.
\end{align*}
The number \begin{math}\ro(n)\end{math} of optimal representations of \begin{math}n\end{math} is equal to \begin{math}u_1(n)\end{math}.
\end{lemma}

A straightforward calculation shows that
\begin{equation}\label{eq:8n_a}
\begin{aligned}
&u_1(8n) = u_2(8n) = \cdots = u_5(8n) = u_1(8n+1) = u_2(8n+1) = u_1(n),\\
&u_3(8n+1) = u_4(8n+1) = u_5(8n+1) = 0.
\end{aligned}
\end{equation}
This gives us the following result (see the full version of this
extended abstract for a detailed proof):

\begin{lemma}\label{lem:optrep}
The number of optimal \begin{math}\{0,1,-1\}\end{math}-representations of a nonnegative integer is a \begin{math}2\end{math}-quasimulti\-plicative function. Specifically, for any three nonnegative integers \begin{math}a,b,k\end{math} with \begin{math}b < 2^k\end{math}, we have
\begin{equation*}\ro(2^{k+3}a + b) = \ro(a)\ro(b).\end{equation*}
\end{lemma}

In Section~\ref{sec:q-regular}, we will show that this is also an instance of a more general phenomenon.

\subsection*{The run length transform and cellular automata}

The \emph{run length transform} of a sequence is defined in a recent paper of Sloane \cite{Sloane:number-on}: it is based on the binary representation, but could in principle also be generalised to other bases. Given a sequence \begin{math}s_1,s_2,\ldots\end{math}, its run length transform is obtained by the rule
\begin{equation*}t(n) = \prod_{i \in \mathcal{L}(n)} s_i,\end{equation*}
where \begin{math}\mathcal{L}(n)\end{math} is the multiset of run lengths of \begin{math}n\end{math} (lengths
of blocks of consecutive ones in the binary representation). For
example, the binary expansion of \begin{math}1910\end{math} is \begin{math}11101110110\end{math}, so the
multiset \begin{math}\mathcal{L}(n)\end{math} of run lengths would be \begin{math}\{3,3,2\}\end{math}, giving
\begin{math}t(1910) = s_2 s_3^2\end{math}.

A typical example is obtained for the sequence of Jacobsthal numbers given by the formula \begin{math}s_n = \frac13 (2^{n+2} - (-1)^n)\end{math}. The associated run length transform \begin{math}t_n\end{math} (sequence \href{http://oeis.org/A071053}{A071053} in the OEIS \cite{OEIS:2016}) counts the number of odd coefficients in the expansion of \begin{math}(1+x+x^2)^n\end{math}, and it can also be interpreted as the number of active cells at the \begin{math}n\end{math}-th generation of a certain cellular automaton. Further examples stemming from cellular automata can be found in Sloane's paper \cite{Sloane:number-on}.

The argument that proved \begin{math}q\end{math}-quasiadditivity of block counts also applies here, and indeed it is easy to see that the identity
\begin{equation*}t(2^{k+1}a + b) = t(a)t(b),\end{equation*}
where \begin{math}0 \leq b < 2^k\end{math}, holds for the run length transform of any sequence, meaning that any such transform is \begin{math}2\end{math}-quasimultiplicative. In fact, it is not difficult to show that every \begin{math}2\end{math}-quasimultiplicative function with parameter \begin{math}r=1\end{math} is the run length transform of some sequence.

\section{Elementary properties}
\label{sec:elem-properties}
Now that we have gathered some motivating examples for the concepts of \begin{math}q\end{math}-quasiadditivity and \begin{math}q\end{math}-quasi\-multiplicativity, let us present some simple results about functions with these properties. First of all, let us state an obvious relation between \begin{math}q\end{math}-quasiadditive and \begin{math}q\end{math}-quasimultiplicative functions:

\begin{prop}\label{prop:trivial}
If a function \begin{math}f\end{math} is \begin{math}q\end{math}-quasiadditive, then the function defined by \begin{math}g(n) = c^{f(n)}\end{math} for some positive constant \begin{math}c\end{math} is \begin{math}q\end{math}-quasimultiplicative. Conversely, if \begin{math}f\end{math} is a \begin{math}q\end{math}-quasimultiplicative function that only takes positive values, then the function defined by \begin{math}g(n) = \log_c f(n)\end{math} for some positive constant \begin{math}c \neq 1\end{math} is \begin{math}q\end{math}-quasiadditive.
\end{prop}

The next proposition deals with the parameter \begin{math}r\end{math} in the definition of a \begin{math}q\end{math}-quasiadditive function:

\begin{prop}
If the arithmetic function \begin{math}f\end{math} satisfies
$f(q^{k+r}a + b) = f(a) + f(b)$
for some fixed nonnegative integer \begin{math}r\end{math} whenever \begin{math}0 \leq b < q^k\end{math}, then it also satisfies
$f(q^{k+s}a + b) = f(a) + f(b)$
for all nonnegative integers \begin{math}s \geq r\end{math} whenever \begin{math}0 \leq b < q^k\end{math}.
\end{prop}

\begin{proof}
If \begin{math}a,b\end{math} are nonnegative integers with \begin{math}0 \leq b < q^k\end{math}, then clearly also \begin{math}0 \leq b < q^{k+s-r}\end{math} if \begin{math}s \geq r\end{math}, and thus
\begin{equation*}f(q^{k+s}a + b) = f(q^{(k+s-r)+r}a + b) = f(a) + f(b).\end{equation*}
\end{proof}

\begin{cor}\label{cor:lin_comb}
If two arithmetic functions \begin{math}f\end{math} and \begin{math}g\end{math} are \begin{math}q\end{math}-quasiadditive functions, then so is any linear combination \begin{math}\alpha f + \beta g\end{math} of the two.
\end{cor}
\begin{proof}
In view of the previous proposition, we may assume the parameter \begin{math}r\end{math} in~\eqref{eq:q-add} to be the same for both functions. The statement follows immediately.
\end{proof}

Finally, we observe that \begin{math}q\end{math}-quasiadditive
and \begin{math}q\end{math}-quasimultiplicative functions can be
computed by breaking the \begin{math}q\end{math}-ary expansion into
pieces. A detailed proof can be found in the full version:

\begin{lemma}\label{lem:simplefacts}
If \begin{math}f\end{math} is a \begin{math}q\end{math}-quasiadditive (\begin{math}q\end{math}-quasimultiplicative) function, then
\begin{itemize}
\item \begin{math}f(0) = 0\end{math} (\begin{math}f(0) = 1\end{math}, respectively, unless \begin{math}f\end{math} is identically \begin{math}0\end{math}),
\item \begin{math}f(qa) = f(a)\end{math} for all nonnegative integers \begin{math}a\end{math}.
\end{itemize}
\end{lemma}

\begin{prop}\label{prop:split}
Suppose that the function \begin{math}f\end{math} is \begin{math}q\end{math}-quasiadditive with parameter \begin{math}r\end{math}, i.e., \begin{math}f(q^{k+r}a + b) = f(a) + f(b)\end{math} whenever \begin{math}0 \leq b < q^k\end{math}. Going from left to right, split the $q$-ary expansion of \begin{math}n\end{math} into blocks by inserting breaks after each run of \begin{math}r\end{math} or more zeros. If these blocks are the $q$-ary representations of \begin{math}n_1,n_2,\ldots,n_{\ell}\end{math}, then we have
\begin{equation*}f(n) = f(n_1) + f(n_2) + \cdots + f(n_{\ell}).\end{equation*}
Moreover, if \begin{math}m_i\end{math} is
the greatest divisor
of \begin{math}n_i\end{math} which
are not divisible by \begin{math}q\end{math} for $i=1,\ldots,\ell$, then
\begin{equation*}f(n) = f(m_1) + f(m_2) + \cdots + f(m_{\ell}).\end{equation*}
Analogous statements hold for \begin{math}q\end{math}-quasimultiplicative functions, with sums replaced by products.
\end{prop}

\begin{proof}
This is obtained by a straightforward induction on \begin{math}\ell\end{math} together with the fact that \begin{math}f(q^{h} a) = f(a)\end{math}, which follows from the previous lemma.
\end{proof}

\begin{example}
Recall that the Hamming weight of the NAF (which is the minimum Hamming weight of a \begin{math}\{0,1,-1\}\end{math}-representation) is \begin{math}2\end{math}-quasiadditive with parameter \begin{math}r=2\end{math}. To determine \begin{math}\hn(314\,159\,265)\end{math}, we split the binary representation, which is
\begin{math}10010101110011011000010100001,\end{math}
into blocks by inserting breaks after each run of at least two zeros:
\begin{equation*}100|101011100|110110000|1010000|1.\end{equation*}
The numbers \begin{math}n_1,n_2,\ldots,n_{\ell}\end{math} in the statement of the proposition are now \begin{math}4,348,432,80,1\end{math} respectively, and the numbers \begin{math}m_1,m_2,\ldots,m_{\ell}\end{math} are therefore \begin{math}1,87,27,5,1\end{math}. Now we use the values \begin{math}\hn(1) = 1\end{math}, \begin{math}\hn(5) = 2\end{math}, \begin{math}\hn(27) = 3\end{math} and \begin{math}\hn(87) = 4\end{math} to obtain
\begin{equation*}\hn(314\,159\,265) = 2\hn(1) + \hn(5) + \hn(27) + \hn(87) = 11.\end{equation*}
\end{example}

\begin{example}
In the same way, we consider the number of optimal representations \begin{math}\ro\end{math}, which is \begin{math}2\end{math}-quasimultiplicative with parameter \begin{math}r=3\end{math}. Consider for instance the binary representation of \begin{math}204\,280\,974\end{math}, namely 
\begin{math}1100001011010001010010001110\end{math}.
We split into blocks:
\begin{equation*}110000|101101000|101001000|1110.\end{equation*}
The four blocks correspond to the numbers \begin{math}48 = 16 \cdot 3\end{math}, \begin{math}360 = 8 \cdot 45\end{math}, \begin{math}328 = 8 \cdot 41\end{math} and \begin{math}14 = 2 \cdot 7\end{math}. Since \begin{math}\ro(3) = 2\end{math}, \begin{math}\ro(45) = 5\end{math}, \begin{math}\ro(41) = 1\end{math} and \begin{math}\ro(7) = 1\end{math}, we obtain
\begin{math}\ro(204\,280\,974) = 10\end{math}.
\end{example}

\section{$q$-Regular functions}\label{sec:q-regular}
In this section, we introduce \begin{math}q\end{math}-regular functions and examine the
connection to our concepts. See~\cite{Allouche-Shallit:2003:autom} for
more background on \begin{math}q\end{math}-regular sequences.

A function \begin{math}f\end{math} is \emph{\begin{math}q\end{math}-regular} if it can be expressed as \begin{math}f=\bfu^{t}\bff\end{math} for a vector \begin{math}\bfu\end{math}
and a vector-valued function \begin{math}\bff\end{math}, and there are matrices \begin{math}M_{i}\end{math}, \begin{math}0\leq i<q\end{math}, satisfying
\begin{equation}\label{eq:q-regular-recursive}
  \bff(qn+i)=M_{i}\bff(n)
\end{equation}
for \begin{math}0\leq i<q\end{math}, \begin{math}qn+i>0\end{math}. We set \begin{math}\bfv=\bff(0)\end{math}.

Equivalently, a function \begin{math}f\end{math} is \begin{math}q\end{math}-regular if and only if \begin{math}f\end{math} can be written as
\begin{equation}
  \label{eq:q-regular}
  f(n)=\bfu^{t} \prod_{i=0}^{L} M_{n_{i}}\bfv
\end{equation}
where \begin{math}n_{L}\cdots n_{0}\end{math} is the \begin{math}q\end{math}-ary expansion of \begin{math}n\end{math}.

The notion of \begin{math}q\end{math}-regular functions is a generalisation of
\begin{math}q\end{math}-additive and \begin{math}q\end{math}-multiplicative functions. However, we emphasise that \begin{math}q\end{math}-quasiadditive and \begin{math}q\end{math}-quasimultiplicative functions are not
necessarily \begin{math}q\end{math}-regular: a \begin{math}q\end{math}-regular sequence can always be bounded
by \begin{math}O(n^{c})\end{math} for a constant \begin{math}c\end{math}, see~\cite[Thm.\
16.3.1]{Allouche-Shallit:2003:autom}. In our setting however, the values of \begin{math}f(n)\end{math} can be chosen arbitrarily for those \begin{math}n\end{math} whose \begin{math}q\end{math}-ary expansion does not contain \begin{math}0^{r}\end{math}. Therefore a \begin{math}q\end{math}-quasiadditive or -multiplicative function can grow arbitrarily fast.

We call \begin{math}(\bfu, (M_{i})_{0\leq i<q}, \bfv)\end{math} a
\emph{linear representation} of the
\begin{math}q\end{math}-regular function \begin{math}f\end{math}. Such
a linear representation is called
\emph{zero-insensitive} if \begin{math}M_{0}\bfv=\bfv\end{math}, meaning that in
\eqref{eq:q-regular}, leading zeros in the \begin{math}q\end{math}-ary expansion of \begin{math}n\end{math} do
not change anything. We call a linear representation \emph{minimal} if the dimension
of the matrices \begin{math}M_{i}\end{math} is minimal among all linear representations of \begin{math}f\end{math}.

 Following \cite{Dumas:2014:asymp}, every \begin{math}q\end{math}-regular function has a
 zero-insensitive minimal linear representation.

\subsection{When is a $q$-regular function $q$-quasimultiplicative?}
We now give a characterisation of \begin{math}q\end{math}-regular
functions that are \begin{math}q\end{math}-quasimultiplicative. Proofs
of the results in this and the following subsection can be found in
the full version.
\begin{theorem}\label{theorem:reg-mult}
  Let \begin{math}f\end{math} be a \begin{math}q\end{math}-regular
  sequence with zero-insensitive minimal linear representation
  \eqref{eq:q-regular}. Then the following
  two assertions are equivalent:
  \begin{itemize}
  \item The sequence \begin{math}f\end{math} is \begin{math}q\end{math}-quasimultiplicative with parameter
    \begin{math}r\end{math}.
    \item    \begin{math}M_{0}^{r}=\bfv\bfu^{t}\end{math}.
  \end{itemize}
\end{theorem}

\begin{example}[The number of optimal \begin{math}\{0,1,-1\}\end{math}-representations]
  The number of optimal \begin{math}\{0,1,-1\}\end{math}-repre\-sentations as described in
  Section~\ref{sec:exampl-q-quasiadd} is a \begin{math}2\end{math}-regular sequence by
  Lemma~\ref{lemma:opt-representations-recursion}. A
  minimal zero-insensitive linear representation for the vector \begin{math}(u_{1}(n),
  u_{2}(n), u_{3}(n), u_{1}(n+1), u_{4}(n+1),
  u_{5}(n+1))^{t}\end{math} is given by
  \begin{equation*}
    M_{0}=
    \begin{pmatrix}
      1&0&0&0&0&0\\
      1&0&0&0&0&0\\
      0&1&0&0&0&0\\
      0&1&0&0&1&0\\
      0&0&0&0&0&1\\
      0&0&0&0&0&0
    \end{pmatrix},\quad
    M_{1}=
    \begin{pmatrix}
      0&1&0&0&1&0\\
      0&0&1&0&0&0\\
      0&0&0&0&0&0\\
      0&0&0&1&0&0\\
      0&0&0&1&0&0\\
      0&0&0&0&1&0
    \end{pmatrix},
  \end{equation*}
\begin{math}\bfu^{t}=(1,0,0,0,0,0)\end{math} and \begin{math}\bfv=(1,1,1,1,0,0)^{t}\end{math}.

As \begin{math}M_{0}^{3}=vu^{t}\end{math}, this sequence is \begin{math}2\end{math}-quasimultiplicative with
parameter \begin{math}3\end{math}, which is the same result as in Lemma~\ref{lem:optrep}. 
\end{example}

\begin{remark}
  The condition on the minimality of the linear representation in
  Theorem~\ref{theorem:reg-mult} is necessary as illustrated by the
  following example:
  
  Consider the sequence \begin{math}f(n)=2^{s_{2}(n)}\end{math} where \begin{math}s_{2}(n)\end{math} is the binary sum of digits
  function. This sequence is \begin{math}2\end{math}-regular and
  \begin{math}2\end{math}-(quasi-)multiplicative with parameter \begin{math}r=0\end{math}. A minimal
 linear representation is given by \begin{math}M_{0}=1\end{math}, \begin{math}M_{1}=2\end{math}, \begin{math}v=1\end{math} and \begin{math}u=1\end{math}. As stated
  in Theorem~\ref{theorem:reg-mult}, we have \begin{math}M_{0}^{0}=vu^{t}=1\end{math}.

  If we use the zero-insensitive non-minimal linear representation defined by \begin{math}M_{0}=\big(
  \begin{smallmatrix}
    1&13\\0&2
  \end{smallmatrix}\big)
\end{math}, \begin{math}M_{1}=\big(
\begin{smallmatrix}
  2&27\\0&5
\end{smallmatrix}
\big)\end{math}, \begin{math}v=(1, 0)^{t}\end{math} and \begin{math}u^{t}=(1, 0)\end{math} instead, we have \begin{math}\rank M_{0}^{r}=2\end{math}
for all \begin{math}r\geq 0\end{math}. Thus \begin{math}M_{0}^{r}\neq vu^{t}\end{math}.
\end{remark}

\subsection{When is a $q$-regular function $q$-quasiadditive?}

The characterisation of \begin{math}q\end{math}-regular functions that are also
\begin{math}q\end{math}-quasiadditive is somewhat more complicated. Again, we consider a
zero-insensitive (but not necessarily minimal) linear representation. We let \begin{math}U\end{math} be the smallest
vector space such that all vectors of the form \begin{math}\bfu^{t}\prod_{i\in I}
M_{n_{i}}\end{math} lie in the affine subspace \begin{math}\bfu^{t} + U^t\end{math} (\begin{math}U^t\end{math} is used
as a shorthand for \begin{math}\{\bfx^{t} \,:\, \bfx \in U\}\end{math}). Such a vector
space must exist, since \begin{math}\bfu^{t}\end{math} is a vector of this form
(corresponding to the empty product, where \begin{math}I = \emptyset\end{math}). Likewise,
let \begin{math}V\end{math} be the smallest vector space such that all vectors of the form
\begin{math}\prod_{j\in J}M_{n_{j}}\bfv\end{math} lie in the affine subspace \begin{math}\bfv +
V\end{math}.

\begin{theorem}\label{thm:q-reg-q-quasiadd}
  Let \begin{math}f\end{math} be a \begin{math}q\end{math}-regular
  sequence with zero-insensitive linear representation
  \eqref{eq:q-regular}. The sequence \begin{math}f\end{math} is \begin{math}q\end{math}-quasiadditive with parameter \begin{math}r\end{math} if and only if all of the following statements hold:
\begin{itemize}
\item \begin{math}\bfu^t \bfv = 0\end{math},
\item \begin{math}U^t\end{math} is orthogonal to \begin{math}(M_0^r - I)\bfv\end{math}, i.e., \begin{math}\bfx^t(M_0^r - I)\bfv = \bfx^tM_0^r\bfv - \bfx^t\bfv = 0\end{math} for all \begin{math}\bfx \in U\end{math},
\item \begin{math}V\end{math} is orthogonal to \begin{math}\bfu^t(M_0^r - I)\end{math}, i.e., \begin{math}\bfu^t(M_0^r - I)\bfy = \bfu^tM_0^r\bfy - \bfu^t\bfy = 0\end{math} for all \begin{math}\bfy \in V\end{math},
\item \begin{math}U^t M_0^r V = 0\end{math}, i.e., \begin{math}\bfx^t M_0^r \bfy  = 0\end{math} for all \begin{math}\bfx \in U\end{math} and \begin{math}\bfy \in V\end{math}.
\end{itemize}
\end{theorem}

\begin{example}
  For the Hamming weight of the nonadjacent form, a zero-insensitive
  (and also minimal) linear representation for the vector \begin{math}(\hn(n),\hn(n+1),\hn(2n+1),1)^{t}\end{math} is
  \begin{equation*}
    M_{0}=
    \begin{pmatrix}
      1&0&0&0\\0&0&1&0\\1&0&0&1\\0&0&0&1
    \end{pmatrix},\quad
    M_{1}=
    \begin{pmatrix}
      0&0&1&0\\0&1&0&0\\0&1&0&1\\0&0&0&1
    \end{pmatrix},
  \end{equation*}
\begin{math}\bfu^{t}=(1,0,0,0)\end{math} and \begin{math}\bfv=(0,1,1,1)^{t}\end{math}.

The three vectors \begin{math}\mathbf{w}_1 = \bfu^{t}M_{1}-\bfu^{t}\end{math},
\begin{math}\mathbf{w}_2 = \bfu^{t}M_{1}^{2}-\bfu^{t}\end{math} and
\begin{math}\mathbf{w}_3 =  \bfu^{t}M_{1}M_{0}M_{1}-\bfu^{t}\end{math} are linearly
independent. If we let \begin{math}W\end{math} be the vector space spanned by those three, it is easily verified that \begin{math}M_{0}\end{math} and \begin{math}M_{1}\end{math} 
map the affine subspace \begin{math}\bfu^{t}+ W^t\end{math} to itself, so \begin{math}U=W\end{math} is spanned by these vectors.

Similarly, the three vectors \begin{math}M_{1}\bfv-\bfv\end{math},
\begin{math}M_{1}^{2}\bfv-\bfv\end{math} and
\begin{math}M_{1}M_{0}M_{1}\bfv-\bfv\end{math} span \begin{math}V\end{math}.

The first condition of Theorem~\ref{thm:q-reg-q-quasiadd} is obviously
true. We only have to verify the other three conditions with \begin{math}r=2\end{math} for the basis vectors
of \begin{math}U\end{math} and \begin{math}V\end{math}, which is done easily. Thus \begin{math}\hn\end{math} is a \begin{math}2\end{math}-regular
sequence that is also \begin{math}2\end{math}-quasiadditive, as was also proved in Section~\ref{sec:exampl-q-quasiadd}.
\end{example}

Finding the vector spaces \begin{math}U\end{math}
and \begin{math}V\end{math} is not trivial. But in a certain special
case of \begin{math}q\end{math}-regular functions, we can give a sufficient condition for
\begin{math}q\end{math}-additivity, which is easier to check. These \begin{math}q\end{math}-regular functions are output sums of
transducers as defined
in~\cite{Heuberger-Kropf-Prodinger:2015:output}: a transducer
transforms the \begin{math}q\end{math}-ary expansion of an integer \begin{math}n\end{math} (read from the least
significant to the most significant digit) deterministically into an output
sequence and leads to a state \begin{math}s\end{math}. The output sum is then the sum of this output sequence
together with the final output of the state \begin{math}s\end{math}. This defines the
value of the
\begin{math}q\end{math}-regular function evaluated at \begin{math}n\end{math}. The function \begin{math}\hn\end{math} discussed in the example above, as well as many other examples, can be represented in this way.

\begin{prop}\label{proposition:q-add-transducer}
  The output sum of a connected transducer is \begin{math}q\end{math}-additive with parameter \begin{math}r\end{math} if the following
  conditions are satisfied:
  \begin{itemize}
  \item The transducer has the reset sequence \begin{math}0^{r}\end{math} going to the
    initial state, i.e., reading
    \begin{math}r\end{math} zeros always leads to the initial state of the transducer.
  \item For every state, the output sum along the path of the reset
    sequence \begin{math}0^{r}\end{math} equals the final output of this state.
  \item Additional zeros at the end of the input sequence do not
    change the output sum.
  \end{itemize}
\end{prop}

\section{A central limit theorem for $q$-quasiadditive and -multiplicative functions}
In this section, we prove a central limit theorem for
\begin{math}q\end{math}-quasimultiplicative functions taking only positive values.
By Proposition~\ref{prop:trivial}, this also implies a central
limit theorem for \begin{math}q\end{math}-quasiadditive functions.

To this end, we define a generating function: let \begin{math}f\end{math} be a \begin{math}q\end{math}-quasimultiplicative function with positive values, let \begin{math}\M_k\end{math} be the set of all nonnegative integers less than \begin{math}q^k\end{math} (i.e., those positive integers whose \begin{math}q\end{math}-ary expansion needs at most \begin{math}k\end{math} digits), and set
\begin{equation*}F(x,t) = \sum_{k \geq 0} x^k \sum_{n \in \M_k} f(n)^t.\end{equation*}
The decomposition of Proposition~\ref{prop:split} now translates
directly to an alternative representation for \begin{math}F(x,t)\end{math}: let \begin{math}\B\end{math} be
the set of all positive integers not divisible by \begin{math}q\end{math} whose \begin{math}q\end{math}-ary representation does not contain the block \begin{math}0^{r}\end{math}, let \begin{math}\ell(n)\end{math} denote the length of the \begin{math}q\end{math}-ary representation of \begin{math}n\end{math}, and define the function \begin{math}B(x,t)\end{math} by
\begin{equation*}B(x,t) = \sum_{n \in \B} x^{\ell(n)} f(n)^t.\end{equation*}
We remark that in the special case where \begin{math}q=2\end{math} and \begin{math}r=1\end{math}, this simplifies greatly to
\begin{equation}\label{eq:q2_r1}
B(x,t) = \sum_{k \geq 1} x^{k} f(2^k-1)^t.
\end{equation}

\begin{prop}\label{prop:gf}
The generating function \begin{math}F(x,t)\end{math} can be expressed as
\begin{equation*}F(x,t) = \frac{1}{1-x} \cdot \frac{1}{1 - \frac{x^r}{1-x} B(x,t)} \Big( 1 + (1+x+\cdots+x^{r-1})B(x,t) \Big) = \frac{1+(1+x+\cdots+x^{r-1})B(x,t)}{1-x-x^rB(x,t)}.\end{equation*}
\end{prop}

\begin{proof}
The first factor stands for the initial sequence of leading zeros, the
second factor for a (possibly empty) sequence of blocks consisting of
an element of \begin{math}\B\end{math} and \begin{math}r\end{math} or more zeros, and the last factor for the
final part, which may be empty or an element of \begin{math}\B\end{math} with up to \begin{math}r-1\end{math} zeros (possibly none) added at the end.
\end{proof}

Under suitable assumptions on the growth of
a \begin{math}q\end{math}-quasiadditive
or \begin{math}q\end{math}-quasimultiplicative function, we can
exploit the expression of Proposition~\ref{prop:gf} to prove a central
limit theorem in the following steps (full proofs can again be found
in the full version).

\begin{defi}
  We say that a function \begin{math}f\end{math} has \emph{at most polynomial growth} if
  \begin{math}f(n)=O(n^{c})\end{math} and \begin{math}f(n) = \Omega(n^{-c})\end{math} for a fixed \begin{math}c\geq
  0\end{math}. We say that \begin{math}f\end{math} has \emph{at most logarithmic growth} if
  \begin{math}f(n)=O(\log n)\end{math}.
\end{defi}

Note that our definition of at most polynomial growth is slightly
different than usual: the extra condition \begin{math}f(n) =
  \Omega(n^{-c})\end{math} ensures that the absolute value
of \begin{math}\log f(n)\end{math} does not grow too fast.

\begin{lemma}\label{lemma:singularity} Assume that the positive, \begin{math}q\end{math}-quasimultiplicative function \begin{math}f\end{math} has at most polynomial growth. 

There exist positive constants \begin{math}\delta\end{math} and \begin{math}\epsilon\end{math} such that
\begin{itemize}
\item \begin{math}B(x,t)\end{math} has radius of convergence \begin{math}\rho(t) > \frac1q\end{math} whenever \begin{math}|t| \leq \delta\end{math}.
\item For \begin{math}|t| \leq \delta\end{math}, the equation \begin{math}x + x^r B(x,t) = 1\end{math} has a complex solution \begin{math}\alpha(t)\end{math} with \begin{math}|\alpha(t)| < \rho(t)\end{math} and no other solutions with modulus \begin{math}\leq (1+\epsilon)|\alpha(t)|\end{math}. 
\item Thus the generating function \begin{math}F(x,t)\end{math} has a simple pole at \begin{math}\alpha(t)\end{math} and no further singularities of modulus \begin{math}\leq (1+ \epsilon)|\alpha(t)|\end{math}. 
\item Finally, \begin{math}\alpha\end{math} is an analytic function of \begin{math}t\end{math} for \begin{math}|t| \leq \delta\end{math}.
\end{itemize}
\end{lemma}

\begin{lemma}\label{lem:sing_anal}
Assume that the positive, \begin{math}q\end{math}-quasimultiplicative function \begin{math}f\end{math} has at most polynomial growth.

With \begin{math}\delta\end{math} and \begin{math}\epsilon\end{math} as in the previous lemma, we have, uniformly in \begin{math}t\end{math},
\begin{equation*}[x^k] F(x,t) = \kappa(t) \cdot \alpha(t)^{-k} \big(1 + O((1+\epsilon)^{-k})\big)\end{equation*}
for some function \begin{math}\kappa\end{math}. Both \begin{math}\alpha\end{math} and \begin{math}\kappa\end{math} are analytic functions of \begin{math}t\end{math} for \begin{math}|t| \leq \delta\end{math}, and \begin{math}\kappa(t) \neq 0\end{math} in this region.
\end{lemma}

\begin{theorem}\label{thm:clt-mult}
Assume that the positive, \begin{math}q\end{math}-quasimultiplicative function \begin{math}f\end{math} has at most polynomial growth.

Let \begin{math}N_k\end{math} be a randomly chosen integer in \begin{math}\{0,1,\ldots,q^k-1\}\end{math}. The
random variable \begin{math}L_k = \log f(N_k)\end{math} has mean \begin{math}\mu k + O(1)\end{math} and
variance \begin{math}\sigma^2 k + O(1)\end{math}, where the two constants are given by
\begin{equation*}\mu = \frac{B_t(1/q,0)}{q^{2r}}\end{equation*}
and
\begin{multline}
  \sigma^2= -B_{t}(1/q,0)^{2} {q}^{-4r+1}(q-1)^{-1} + 2B_{t}(1/q,0)^{2} {q}^{-3r+1}(q-1)^{-1} -B_{t}(1/q,0)^{2}{q}^{-4r}(q-1)^{-1} \\-
     4rB_{t}(1/q,0)^{2} {q}^{-4r} + B_{tt}(1/q,0){q}^{-2r}
     - 2B_{t}(1/q,0)
        B_{tx}(1/q,0)
{q}^{-4r-1} .
\end{multline}
If \begin{math}f\end{math} is not the constant function \begin{math}f \equiv 1\end{math}, then \begin{math}\sigma^2 \neq 0\end{math} and the normalised random variable \begin{math}(L_k - \mu k)/(\sigma \sqrt{k})\end{math} converges weakly to a standard Gaussian distribution.
\end{theorem}

\begin{cor}\label{cor:clt-add}
  Assume that the \begin{math}q\end{math}-quasiadditive function \begin{math}f\end{math} has at most logarithmic growth.

Let \begin{math}N_k\end{math} be a randomly chosen integer in \begin{math}\{0,1,\ldots,q^k-1\}\end{math}. The
random variable \begin{math}L_k = f(N_k)\end{math} has mean \begin{math}\hat\mu k + O(1)\end{math} and
variance \begin{math}\hat\sigma^2 k + O(1)\end{math}, where the two constants \begin{math}\mu\end{math} and \begin{math}\sigma^2\end{math}are given by
%\begin{equation*}\hat\mu = \frac{\hat B_t(1/q,0)}{q^{2r}}\end{equation*}
%and
%\begin{multline*}
%  \hat \sigma^2= -\hat B_{t}(1/q,0)^{2} {q}^{-4r+1}(q-1)^{-1} + 2\hat
%  B_{t}(1/q,0)^{2} {q}^{-3r+1}(q-1)^{-1} -\hat B_{t}(1/q,0)^{2}{q}^{-4r}(q-1)^{-1}\\ -
%     4r\hat B_{t}(1/q,0)^{2} {q}^{-4r} + \hat B_{tt}(1/q,0){q}^{-2r}
%     - 2\hat B_{t}(1/q,0)
%        \hat B_{tx}(1/q,0)
%{q}^{-4r-1},
%\end{multline*}
the same formulas as in Theorem~\ref{thm:clt-mult}, with \begin{math}B(x,t)\end{math} replaced by
\begin{equation*}
  \hat B(x,t) = \sum_{n \in \B} x^{\ell(n)} e^{f(n)t}.
\end{equation*}

If \begin{math}f\end{math} is not the constant function \begin{math}f \equiv 0\end{math}, then the normalised random variable \begin{math}(L_k - \hat\mu k)/(\hat\sigma \sqrt{k})\end{math} converges weakly to a standard Gaussian distribution.
\end{cor}

\begin{remark}
By means of the Cram\'er-Wold device (and Corollary~\ref{cor:lin_comb}), we also obtain joint normal distribution of tuples of \begin{math}q\end{math}-quasiadditive functions.
\end{remark}

We now revisit the examples discussed in
Section~\ref{sec:exampl-q-quasiadd} and state the corresponding
central limit theorems. Some of them are well known while others are
new. We also provide numerical values for the constants in mean and variance.
\begin{example}[see also \cite{Kirschenhofer:1983:subbl,Drmota:2000}]The number of blocks \begin{math}0101\end{math} occurring in the binary
  expansion of \begin{math}n\end{math} is a \begin{math}2\end{math}-quasiadditive function of at most
  logarithmic growth. Thus by Corollary~\ref{cor:clt-add}, the
  standardised random variable is asymptotically normally distributed, the constants being \begin{math}\hat\mu = \frac1{16}\end{math} and \begin{math}\hat\sigma^2 = \frac{17}{256}\end{math}.
\end{example}

\begin{example}[see also \cite{Thuswaldner:1999,Heuberger-Kropf:2013:analy}]
  The Hamming weight of the nonadjacent form is \begin{math}2\end{math}-quasiadditive
  with at most logarithmic growth (as the length of the NAF of \begin{math}n\end{math} is logarithmic). Thus by Corollary~\ref{cor:clt-add}, the
  standardised random variable is asymptotically normally distributed. The associated constants are \begin{math}\hat\mu = \frac13\end{math} and \begin{math}\hat\sigma^2 = \frac2{27}\end{math}.
\end{example}

\begin{example}[see Section~\ref{sec:exampl-q-quasiadd}]
  The number of optimal \begin{math}\{0,1,-1\}\end{math}-representations is
  \begin{math}2\end{math}-quasimultiplicative. As it is always greater or equal to \begin{math}1\end{math} and
  \begin{math}2\end{math}-regular, it has at most polynomial growth. Thus
  Theorem~\ref{thm:clt-mult} implies that the standardised logarithm
  of this random variable is asymptotically normally distributed with
  numerical constants given by \begin{math}\mu\approx 0.060829\end{math}, \begin{math}\sigma^{2}\approx 0.038212\end{math}.
\end{example}

\begin{example}[see Section~\ref{sec:exampl-q-quasiadd}] Suppose that the sequence \begin{math}s_1,s_2,\ldots\end{math} satisfies \begin{math}s_{n}\geq 1\end{math} and \begin{math}s_{n}=O(c^{n})\end{math} for a constant
  \begin{math}c\geq 1\end{math}.
  The run length transform \begin{math}t(n)\end{math} of \begin{math}s_{n}\end{math}
  is \begin{math}2\end{math}-quasimultiplicative. As \begin{math}s_{n}\geq 1\end{math} for all \begin{math}n\end{math}, we have \begin{math}t(n)\geq
  1\end{math} for all \begin{math}n\end{math} as well. Furthermore, there exists a constant \begin{math}A\end{math} such that \begin{math}s_n \leq A c^n\end{math} for all \begin{math}n\end{math}, and the sum of all run lengths is bounded by the length of the   binary expansion, thus
  \begin{equation*}
t(n)=\prod_{i\in\mathcal L(n)}s_{i} \leq \prod_{i \in \mathcal{L}(n)} (A c^i) \leq (Ac)^{1+\log_2 n}.
\end{equation*}
Consequently, \begin{math}t(n)\end{math} is positive and has at most polynomial growth. By
Theorem~\ref{thm:clt-mult}, we obtain an asymptotic normal
distribution for the standardised random variable \begin{math}\log t(N_{k})\end{math}. The constants \begin{math}\mu\end{math} and \begin{math}\sigma^2\end{math} in mean and variance are given by
\begin{equation*}
\mu = \sum_{i \geq 1} (\log s_i) 2^{-i-2} 
\end{equation*}
and
\begin{equation*}
\sigma^2 = \sum_{i \geq 1} (\log s_i)^2 \big(2^{-i-2} - (2i-1)2^{-2i-4} \big) - \sum_{j > i \geq 1} (\log s_i)(\log s_j)   (i+j-1) 2^{-i-j-3}.
\end{equation*}
These formulas can be derived from those given in Theorem~\ref{thm:clt-mult} by means of the representation~\eqref{eq:q2_r1}, and the terms can also be interpreted easily: write \begin{math}\log t(n) = \sum_{i \geq 1} X_i(n) \log s_i\end{math}, where \begin{math}X_i(n)\end{math} is the number of runs of length \begin{math}i\end{math} in the binary representation of \begin{math}n\end{math}. The coefficients in the two formulas stem from mean, variance and covariances of the \begin{math}X_i(n)\end{math}.

In the special case that
\begin{math}s_{n}\end{math} is the Jacobsthal sequence ($s_n = \frac13(2^{n+2} - (-1)^n)$, see Section~\ref{sec:exampl-q-quasiadd}), we have the
  numerical values
  \begin{math}\mu \approx 0.429947\end{math}, \begin{math}\sigma^{2} \approx 0.121137\end{math}.
\end{example}

\bibliographystyle{amsplain}
\bibliography{lit}
\end{document}